\documentclass[11pt,a4paper]{article}

\usepackage{amsmath,amsfonts,amssymb,latexsym,graphics,epsfig,url}
\usepackage{color}
\usepackage{amsthm}
\usepackage[english]{babel}
\usepackage{graphicx}

\newcommand\blfootnote[1]{%
  \begingroup
  \renewcommand\thefootnote{}\footnote{#1}%
  \addtocounter{footnote}{-1}%
  \endgroup
}

\newtheorem{theorem}{Theorem}[section]
\newtheorem{lemma}[theorem]{Lemma}
\newtheorem{corollary}[theorem]{Corollary}
\newtheorem{proposition}[theorem]{Proposition}

\DeclareMathOperator{\tr}{tr}

\DeclareMathOperator{\spec}{sp}

\def\Fou{{\cal F}}

\def\Z{\ns Z}

\def\a{\mbox{\boldmath $a$}}
\def\b{\mbox{\boldmath $b$}}

\def\x{\mbox{\boldmath $x$}}
\def\y{\mbox{\boldmath $y$}}
\def\z{\mbox{\boldmath $z$}}

\def\vec0{\mbox{\boldmath $0$}}

\def\A{\mbox{\boldmath $A$}}
\def\B{\mbox{\boldmath $B$}}
\def\C{\mbox{\boldmath $C$}}

\def\F{\mbox{\boldmath $F$}}
\def\G{\Gamma}

\def\I{\mbox{\boldmath $I$}}

\def\S{\mbox{\boldmath $S$}}

\def\Y{\mbox{\boldmath $Y$}}
\def\Z{\ns{Z}}
\def\I{\mbox{\boldmath $I$}}

\def\vecZ{\mbox{\boldmath $Z$}}

\def\C{\mathbb C}

\def\G{\Gamma}
\def\Na{\mathbb N}
\def\Re{\mathbb R}
\def\Z{\mathbb Z}

\begin{document}

\title{An algebraic approach to lifts of digraphs
\thanks{Research of the first two authors is supported by MINECO under project MTM2014-60127-P, and by AGAUR under project 2014SGR1147.
The fifth author acknowledges support from the research grants APVV 0136/12,  APVV-15-0220, VEGA 1/0026/16, and VEGA 1/0142/17. }}
\author{C. Dalf\'o$^a$, M. A. Fiol$^b$, M. Miller$^c$, J. Ryan$^d$, J. \v{S}ir\'a\v{n}$^e$\\
$^{a}${\small Departament de Matem\`atiques} \\
{\small Universitat Polit\`ecnica de Catalunya,} 
{\small Barcelona, Catalonia} \\
{\small {\tt{cristina.dalfo@upc.edu}}} \\
$^{b}${\small Departament de Matem\`atiques} \\
{\small Universitat Polit\`ecnica de Catalunya} \\
{\small Barcelona Graduate School of Mathematics,} 
{\small  Barcelona, Catalonia} \\
{\small{\tt{miguel.angel.fiol@upc.edu}}} \\
$^c${\small School of Mathematical and Physical Sciences} \\
{\small The University of Newcastle,} 
{\small Newcastle, Australia}\\
$^c${\small Department of Mathematics} \\
{\small University of West Bohemia,} 
{\small Plze\v{n}, Czech Republic} \\
$^d${\small School of Electrical Engineering and Computing}\\
{\small The University of Newcastle, }
{\small Newcastle, Australia}\\
{\small{\tt {joe.ryan@newcastle.edu.au}}}\\
$^e${\small Department of Mathematics and Statistics}\\
{\small The Open University, Milton Keynes, UK}\\
$^e${\small Department of Mathematics and Descriptive Geometry}\\
{\small Slovak University of Technology, Bratislava, Slovak Republic}\\
{\small{\tt {j.siran@open.ac.uk}}}
}

\maketitle
\vskip -1cm

\blfootnote{
\begin{minipage}[l]{0.3\textwidth} \includegraphics[trim=10cm 6cm 10cm 5cm,clip,scale=0.15]{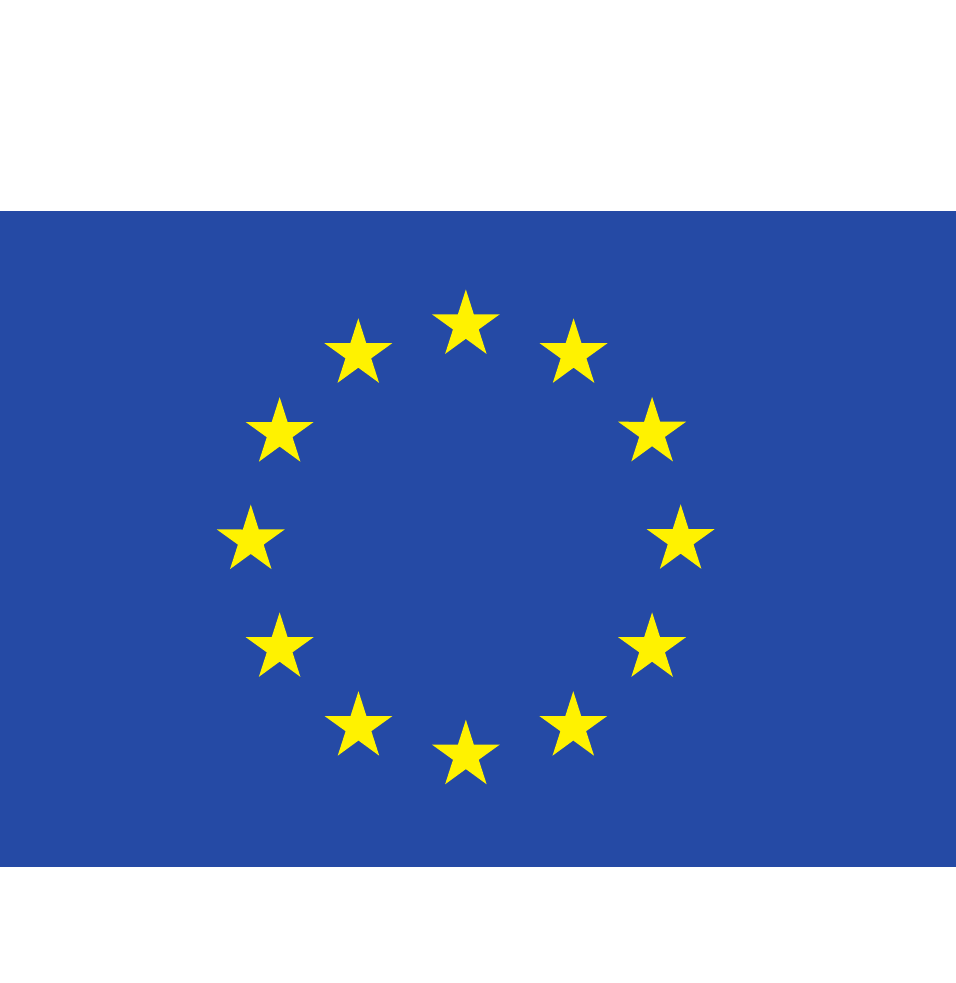} \end{minipage}  \hspace{-2cm} \begin{minipage}[l][1cm]{0.79\textwidth}
   The first author has also received funding from the European Union's Horizon 2020 research and innovation programme under the Marie Sk{\l}odowska-Curie grant agreement No 734922.
  \end{minipage}}

\begin{center}
This paper is dedicated to the memory of Mirka Miller,\\
who enjoyed polynomial matrices very much.
\end{center}

\newpage

\begin{abstract}
We study the relationship between two key concepts in the theory of (di)graphs: the quotient digraph, and the lift $\Gamma^\alpha$ of a base (voltage) digraph. These techniques contract or expand a given digraph in order to study its characteristics, or obtain more involved structures.
This study is carried out by introducing a quotient-like matrix,  with complex polynomial entries, which fully represents $\Gamma^\alpha$. In particular, such a matrix gives the quotient matrix of a regular partition of $\Gamma^\alpha$, and
when the involved group is Abelian,
it completely determines the spectrum of $\Gamma^\alpha$.
As some examples of our techniques, we study some basic properties of the Alegre digraph. In addition we completely characterize the spectrum of a new family of digraphs, which contains the generalized Petersen graphs, and that of the Hoffman-Singleton graph.
\end{abstract}



\section{Introduction}
In the study of interconnection and communication networks, the theory of digraphs plays a key role as, in many cases, the links between nodes are unidirectional.
Within this theory, there are two concepts that have shown to be very fruitful to construct good and efficient networks. Namely, those of quotient digraphs and lifts of voltage digraphs.
Roughly speaking, quotient digraphs allow us to give a simplified or `condensed' version of a larger digraph, while the voltage digraph technique does the converse, by `expanding' a smaller digraph into its `lift'. From this point of view, it is natural that both techniques have close relationships. In this paper we explore some of such interrelations.

The paper is organized as follows. In the rest of this section, we give some basic background information. In Section \ref{sec:reg-part} we present the basic definition and results on regular partitions and their corresponding quotient digraphs.
In  Section~\ref{sec:voltage}, we recall the definitions of voltage and lifted digraphs.
Section \ref{sec:matrixrep} is devoted to studying a representation of a
lifted digraph with a quotient-like matrix whose size equals the order of the (much smaller) base digraph. In particular, it is shown that such a matrix can be used to deduce combinatorial properties of the lifted digraph. Following this approach, and as a main result, Section 5 present a new method to completely determine the spectrum of the lift by using only the spectrum of the quotient digraph. The results are illustrated by following some examples. The first one is the so-called Alegre digraph (first shown by Fiol, Alegre, and Yebra in \cite{FiYeAl84}), which is the largest digraph (with order 25) with degree 2 and diameter 4. This digraph can be constructed as the lifted digraph of a voltage digraph, and a major part of its structure is characteristic of a line digraph.
The second example is a new family of digraphs which, as a particular case, contains the well-know
generalized Petersen graphs. Finally, we recalculate the spectrum of the Hoffman-Singleton graph without using its  strong regularity character.

\subsection{Background}
Here, we recall some basic terminology and simple results concerning digraphs and their
spectra. For the concepts and/or results not presented here, we refer the reader to some of
the basic textbooks on the subject; for instance, Bang-Jensen and Gutin \cite{bg09}, Chartrand and
Lesniak~\cite{cl96}, or Diestel~\cite{d10}.

Through this paper, $\Gamma=(V,E)$ denotes a digraph, with vertex set $V$ and arc set $E$, that is
\emph{strongly connected}, namely, each vertex is connected to all other vertices by traversing the arcs in their corresponding direction. An arc from vertex $u$ to vertex $v$ is denoted by either $(u,v)$, $uv$, or $u\rightarrow v$. We allow {\em loops} (that is, arcs from a vertex to itself), and {\em multiple arcs}. The set of vertices adjacent to and from $v\in V$ is denoted by $\G^{-}(v)$ and $\G^{+}(v)$, respectively. Such vertices are referred to as {\em in-neighbors} and {\em out-neighbors} of $v$, respectively. Moreover, $\delta^-(v)=|\G^{-}(v)|$ and $\delta^+(v)=|\G^{+}(v)|$ are the \emph{in-degree} and \emph{out-degree} of vertex $v$, and $\G$ is {\em $d$-regular} when $\delta^+(v)=\delta^-(v)=d$ for all $v\in V$. Similarly, given $U\subset V$, $\G^{-}(U)$ and $\G^{+}(U)$ represent the sets of vertices adjacent to and from (the vertices) of $U$, respectively. Given two vertex subsets $X,Y\subset V$, the subset of arcs from $X$ to $Y$ is denoted by $e(X,Y)$.

The spectrum of $\Gamma$, denoted by $\spec \Gamma=\{\lambda_0^{m_0},\lambda_1^{m_1},\ldots,\lambda_d^{m_d}\}$, is constituted by the distinct
eigenvalues $\lambda_i$ with the corresponding algebraic multiplicities $m_i$, for $i=0,1,\ldots,d$, of its adjacency matrix $\A$.

\section{Regular partitions and quotient digraphs}
\label{sec:reg-part}
Let $\Gamma=(V,E)$ be a digraph with $n$ vertices and adjacency matrix $\A$. A partition $\pi$ of its vertex set $V=U_1\cup U_2 \cup\cdots\cup U_m$, for $m\le n$, is called {\em regular} if the number of
arcs from a vertex $u\in U_i$ to vertices in $U_j$ only depends on $i$ and $j$. Let $c_{ij}$ be the
number of arcs that join a fixed vertex in $U_i$ to vertices in $U_j$. A matrix characterization of
this property is the following: Let $\S$ be the  $0$-$1$ $(n\times m)$-matrix whose $i$-th
column is the normalized characteristic vector of $U_i$, so that $\S^{\top}\S=\I$ (the identity
matrix), and consider the so-called {\em quotient matrix}
\begin{equation}
\label{B=S^TAS}
\B=\S^{\top}\A\S.
\end{equation}
Then, it can be easily checked that $\pi$ is regular if and only if \begin{equation}
\label{SB=AS}
\S\B=\A\S.
\end{equation}
The digraph $\pi(\G)$ whose (weighted) adjacency matrix is the quotient matrix is called \emph{quotient digraph}, and its arcs can have weight different from 1. More precisely, the vertices of the quotient
digraph are the subsets $U_i$, for $i=1,2,\ldots,m$, and the arc from vertex $U_i$ to vertex $U_j$ has
weight $c_{ij}$.
For the case of quotient digraphs obtained from non-directed graphs, see Godsil~\cite[Lemma 2.1]{g93}.
We have the following basic result, where the regular partition of $V$ is also called a {\it
regular} partition of $\A$.

\begin{lemma}
\label{lema-sp}
Every eigenvalue of the quotient matrix $\B$ of a regular partition of $\A$ is also an eigenvalue of
$\A$, that is, $\spec \B\subset\spec \A$.
\end{lemma}

\begin{figure}[t]
\begin{center}
\includegraphics[width=12cm]{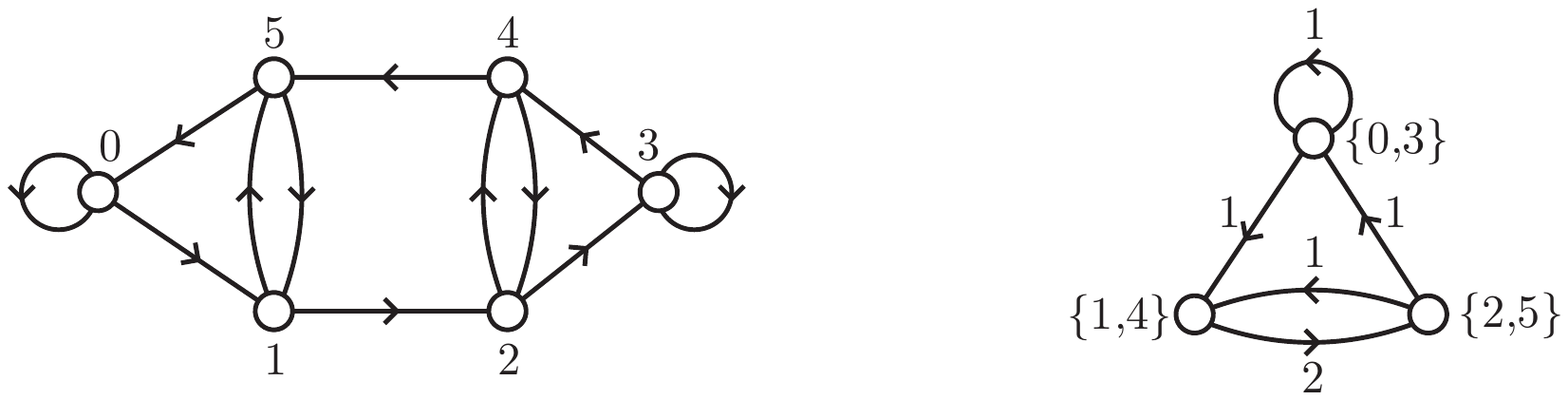}
\end{center}
\vskip-14.5cm
\caption{The digraph with adjacency matrix given by~(\ref{eq:matriu-exemple}), and its quotient digraph with weighted arcs.}
  \label{fig1}
\end{figure}

By way of example, consider the digraph of Figure~\ref{fig1}.
Then, its adjacency matrix is
\begin{equation}
\label{eq:matriu-exemple}
\A=
\left(
\begin{array}{cccccc}
1 & 1 & 0 & 0 & 0 & 0 \\
0 & 0 & 1 & 0 & 0 & 1 \\
0 & 0 & 0 & 1 & 1 & 0 \\
0 & 0 & 0 & 1 & 1 & 0 \\
0 & 0 & 1 & 0 & 0 & 1 \\
1 & 1 & 0 & 0 & 0 & 0
\end{array}
\right),
\end{equation}
with $\spec \A=\{-1,0^{(3)},1,2\}$.
A possible regular partition is constituted by the vertex subsets $U_i=\{i,i+3\}$, for
$i=0,1,2$, with characteristic matrix
$$
\S=\frac{1}{\sqrt{2}}
\left(
\begin{array}{ccc}
1 & 0 & 0 \\
0 & 1 & 0 \\
0 & 0 & 1 \\
1 & 0 & 0 \\
0 & 1 & 0 \\
0 & 0 & 1
\end{array}
\right).
$$
Then the corresponding quotient matrix is
$$
\B=\S^{\top}\A\S=\left(
\begin{array}{ccc}
1 & 1 & 0 \\
0 & 0 & 2 \\
1 & 1 & 0
\end{array}
\right),
$$
with $\spec \B=\{-1,0,2\}\subset \spec \A$, as expected.

\section{Voltage and lifted digraphs}
\label{sec:voltage}

Voltage (di)graphs are, in fact, a type of compounding that consists of connecting
together several copies of a (di)graph by setting some (directed) edges between any
two copies. Usually, the symmetry of the obtained constructions  yields digraphs
with large automorphism groups. As far as we know, one of the first papers where
voltage graphs were used for construction of dense graphs is Alegre, Fiol and
Yebra~\cite{afy86}, but without using the name of `voltage graphs'. This name was
coined previously by Gross~\cite{g74}. For more information, see Gross and
Tucker~\cite{gt87}, Baskoro, Brankovi\'{c}, Miller, Plesn\'{\i}k, Ryan and
Sir\'{a}\v{n}~\cite{bbmrs97}, and Miller and Sir\'{a}\v{n}~\cite{ms}.

Let $\G$ be a digraph with vertex set $V=V(\G)$ and arc set $E=E(\G)$.
Then, given a group $G$ with generating set $\Delta$, a
{\em voltage assignment} of $\G$ is a mapping $\alpha:E\rightarrow \Delta$. The pair $(\G,\alpha)$ is often called a {\em voltage digraph}. The {\em lifted digraph} (or, simply, {\em lift})
$\Gamma^\alpha$ is the digraph with vertex set  $V(\Gamma^\alpha)=V\times G$ and
arc set $E(\Gamma^\alpha)=E\times G$, where there is an arc from vertex $(u,g)$ to
vertex $(v,h)$ if and only if $uv\in E$ and $h=g\alpha(uv)$. Such an arc is denoted
by $(uv,g)$.

\subsection{The adjacency matrix of the lifted digraph}
It is clear that the base digraph with the voltage assignment univocally determines
the adjacency matrix of its lift.
To define it we need to consider the following concepts.
Given a (multiplicative) group $G$ together with a given ordering of its elements
$g_0(=1), g_1,\ldots, g_{n-1}$, a {\em $G$-circulant matrix} is defined as a square matrix $\A$ of order $n$ indexed by elements of $G$, with first row
$a_{0,0}=a_{g_0}$, $a_{0,1}=a_{g_1}$, $\ldots$, $a_{0,n-1}=a_{g_{n-1}}$, and entries
$$
(\A)_{g,h}=a_{hg}, \qquad g,h\in G.
$$
Thus, the elements of row $g$ are identical to those of the first row, but they are permuted by the action of $g$ on $G$.
In particular, a {\em circulant matrix} (see Davis~\cite{d79}) corresponds to a $G$-circulant matrix with the cyclic group $G=\Z_n$ and natural ordering $0,1,\ldots,n-1$. Another example is the adjacency matrix $\A$ of the Cayley digraph $\textrm{Cay}(G,\Delta)$
of the group $G$ with generating set $\Delta$, which is a $G$-circulant matrix whose first row has
elements $a_{1,j}=1$ if $g_j\in \Delta$, and $a_{1,j}=0$, otherwise.
The concept of block $G$-circulant matrix is similar, but now the
elements  $a_{g_0}$, $a_{g_1}$,\ldots, $a_{g_{n-1}}$ of the first row (and,
consequently, the other rows) are replaced by the $m\times m$ matrices (or blocks)
$\A_0=\A_{g_0}$,\ldots, $\A_{n-1}=\A_{g_{n-1}}$.

The following result is an easy consequence of the above definitions.
\begin{lemma}
Let $\G$ be a base digraph with voltage assignment $\alpha$ on the group
$G=\{g_0(=1),\ldots,g_{m-1}\}$.
Let $\G_i$ be the spanning subgraph of $\G$ with arc set $\alpha^{-1}(g_i)$, and adjacency matrix
$\A_i$, for $i=0,\ldots,m-1$. Then, the adjacency matrix $\A$ of the lifted  digraph $\G^{\alpha}$ is the block $G$-circulant matrix with first block-row $\A_0,\A_1,\ldots,\A_{m-1}$.
Moreover, the lift $\G^{\alpha}$ has a regular partition with quotient matrix
$\B=\sum_{i=0}^{m-1} \A_i$. \hfill $\Box$
\end{lemma}

By way of example, let us consider the Alegre digraph, which is a $2$-regular digraph with $n=25$ vertices and diameter $k=4$ represented in Figure~\ref{fig2} (left). This digraph was found by Fiol, Yebra, and Alegre
in~\cite{FiYeAl84}. The Alegre digraph can be seen as the lifted digraph $\Gamma^\alpha$ of the base
digraph $\Gamma$ with the voltage assignments shown in Figure~\ref{fig2} (right).
Then, the nonzero blocks of the first row constituting the adjacency matrix of $\Gamma^\alpha$ are
$$
\A_0=
\left(
\begin{array}{ccccc}
0 & 1 & 1 & 0 & 0 \\
0 & 0 & 0 & 0 & 0 \\
0 & 0 & 0 & 0 & 0 \\
1 & 0 & 0 & 0 & 1 \\
0 & 0 & 0 & 0 & 0
\end{array}
\right),\quad
\A_1=
\left(
\begin{array}{ccccc}
0 & 0 & 0 & 0 & 0 \\
0 & 0 & 0 & 0 & 0 \\
0 & 0 & 0 & 1 & 0 \\
0 & 0 & 0 & 0 & 0 \\
1 & 0 & 0 & 0 & 1
\end{array}
\right),
\quad
\A_4=
\left(
\begin{array}{ccccc}
0 & 0 & 0 & 0 & 0 \\
0 & 1 & 1 & 0 & 0 \\
0 & 0 & 0 & 1 & 0 \\
0 & 0 & 0 & 0 & 0 \\
0 & 0 & 0 & 0 & 0
\end{array}
\right).
$$

\begin{figure}[t]
\begin{center}
\includegraphics[width=14cm]{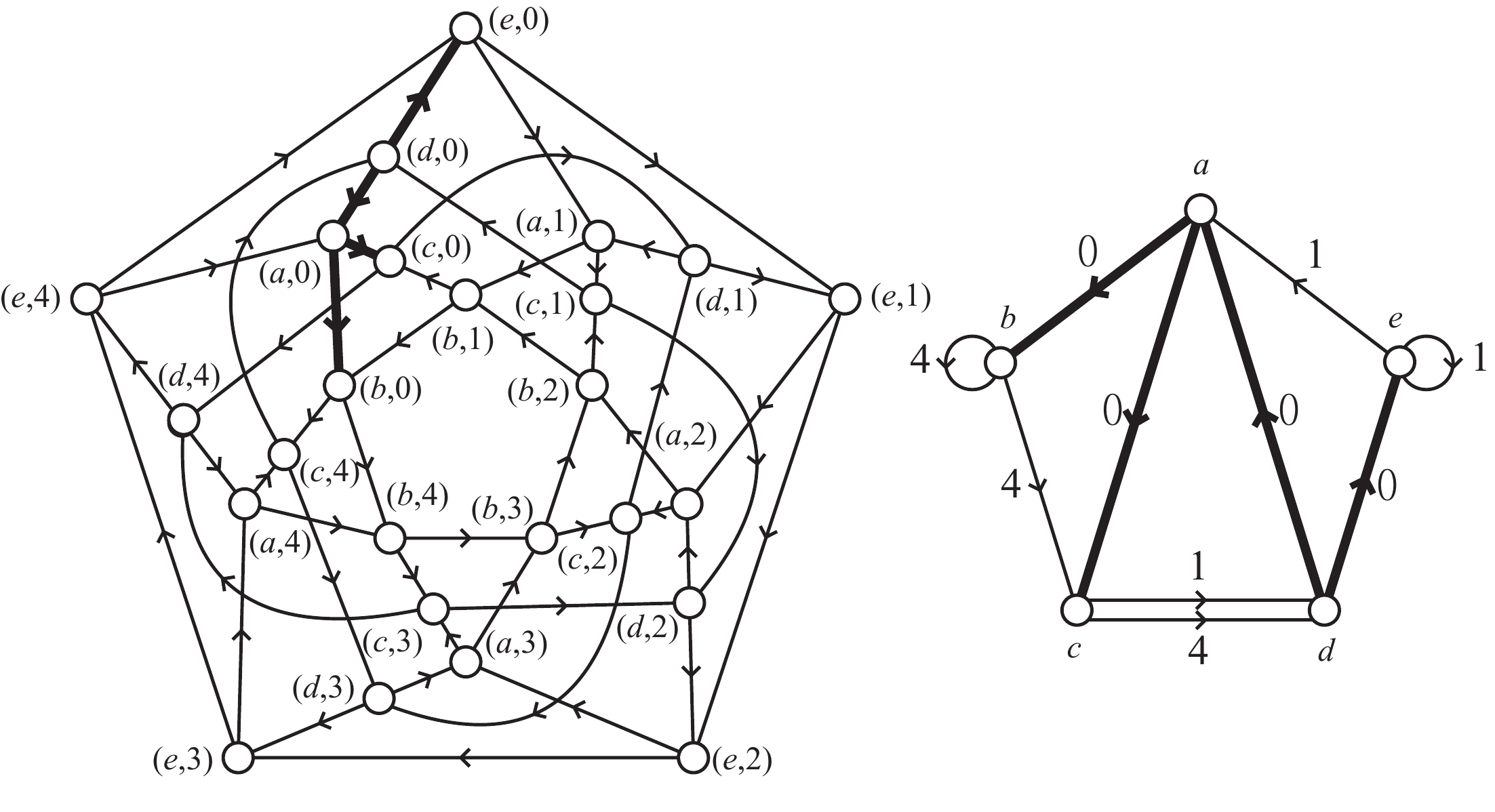}
 \end{center}
 \vskip-.5cm
\caption{The Alegre digraph (left), and its base digraph (right).
The adjacencies in the copy $0$
are represented with a thick line.}
\label{fig2}
\end{figure}

Thus, $\Gamma^\alpha$ has a regular partition with five sets of five vertices each, which is better illustrated in the drawing of Figure \ref{fig3}, with quotient matrix
\begin{equation}
\label{quotient-matrix-Alegre}
\B=\sum_{i=0}^4 \A_i=
\left(
\begin{array}{ccccc}
0 & 1 & 1 & 0 & 0 \\
0 & 1 & 1 & 0 & 0 \\
0 & 0 & 0 & 2 & 0 \\
1 & 0 & 0 & 0 & 1 \\
1 & 0 & 0 & 0 & 1
\end{array}
\right),
\end{equation}

\begin{figure}[t]
\begin{center}
\includegraphics[width=10cm]{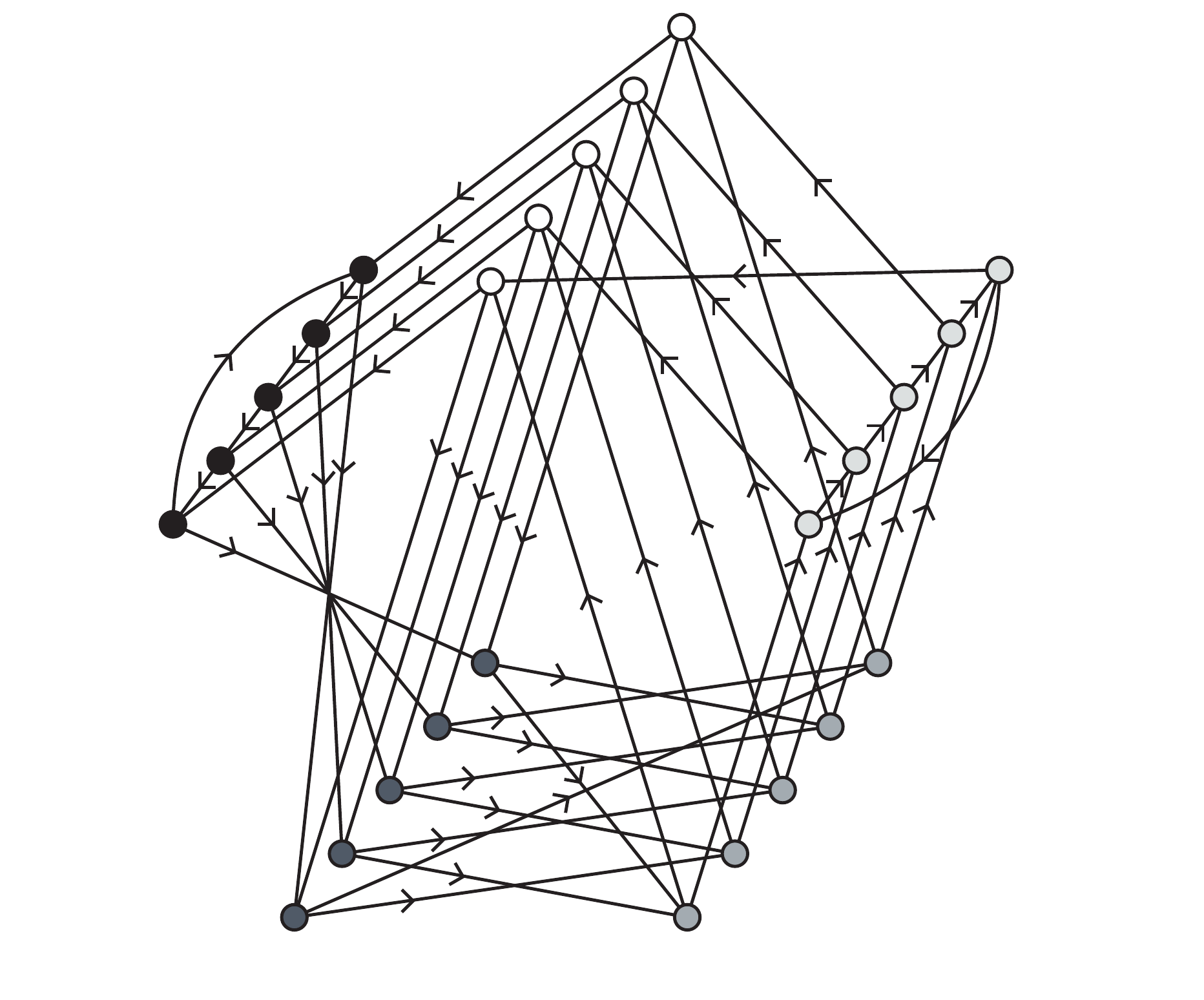}
\end{center}
\vskip -.75cm
\caption{The Alegre digraph and its regular partition with five vertex sets.}
\label{fig3}
\end{figure}

\section{Matrix representations}
\label{sec:matrixrep}
In this section, we study how to fully represent a lifted digraph with a matrix whose size equals the order of the base digraph.

First, we deal with the case when the group $G$ of the voltage assignments is cyclic. Thus, let $\G=(V,E)$ be a digraph with voltage assignment $\alpha$ on the group $G=\Z_k=\{0,1,\ldots,k-1\}$. Its
\emph{polynomial matrix} $\B(z)$ is a square matrix indexed by the vertices of $\G$, and whose elements are complex valuated polynomials in
the quotient ring $\Re_{k-1}[z]=\Re[x]/(z^k)$, where $(z^k)$ is the ideal generated by the polynomial
$z^k$. More precisely, each entry of $\B(z)$ is fully represented by a polynomial of degree at most $k-1$, say $(\B(z))_{uv}=p_{uv}(z)=\alpha_0+\alpha_1 z+\cdots +\alpha_{k-1}z^{k-1}$, where
$$
\alpha_i=\left\{
\begin{array}{ll}
1 & \mbox{if  $uv\in E$ and $\alpha(uv)=i$,}\\
0 & \mbox{otherwise,}
\end{array}
\right.
$$
for $i=0,\ldots,k-1$.
For example, in the case of the Alegre digraph in Figure~\ref{fig1}, the polynomial matrix
is
\begin{equation}
\label{B(z)Alegre}
\B(z)=
\left(
\begin{array}{ccccc}
0 & 1 & 1 & 0 & 0 \\
0 & z^4 & z^4 & 0 & 0 \\
0 & 0 & 0 & z+z^4 & 0 \\
1 & 0 & 0 & 0 & 1 \\
z & 0 & 0 & 0 & z
\end{array}
\right),
\end{equation}
where $(\B(z))_{ij}=\alpha_r z^r+\alpha_s z^s$, with $\alpha_r,\alpha_s\in\{0,1\}$ and
$r,s\in \Z_5$, means that there are arcs from vertex $(i,p)$ to vertices $(j,p+r)$
or/and $(j,p+s)$ if and only if $\alpha_r=1$ or/and $\alpha_s=1$.

Since the polynomial matrix $\B(z)$ fully represents the digraph $\G^{\alpha}$, we should be able to retrieve from the former any property of the latter.
As a first example,
the following result shows that the powers of $\B(z)$ yield the same information as the powers of the adjacency matrix of the lift $\G^{\alpha}$.
\begin{lemma}
\label{lema-walks-cyclic}
Let $(\B(z)^\ell)_{uv}=\beta_0+\beta_1z+\cdots +\beta_{k-1}z^{k-1}$. Then, for every $i=0,\ldots,k-1$,
the coefficient $\beta_i$ equals the number of walks of length $\ell$ in the lifted digraph
$\G^{\alpha}$,  from vertex $(u,h)$ to vertex $(v,h+i)$ for every $h\in G$.
In particular, $\B(1)$ is the quotient matrix of the corresponding regular partition of $\G^{\alpha}$.
\end{lemma}
\begin{proof}
The result is clear for $\ell=0,1$. Then, the result follows easily by using induction.
\end{proof}


For instance, in the case of Alegre digraph, we get
\begin{equation*}
\label{B(z)^4}
\B(z)^4\!=\!
\left(
\begin{array}{ccccc}
2+z^2+z^3 & z+z^2+z^4 & z+z^2+z^4 & z^2+z^4 & 2+z^2+z^3 \,\\
z+z^2+2z^4 & 1+z+z^3 & 1+z+z^3 & z+z^3 & z+z^2+2z^4 \,\\
z+z^3 & 2+z^2+z^3 & 2+z^2+z^3 & 2+z^2+z^3 & z+z^3 \,\\
z+z^3+z^4 & 1+z^2+z^3 &  1+z^2+z^3 & 2+z^2+z^3 & z+z^3+z^4 \,\\
1+z^2+z^4 & z+z^3+z^4 & z+z^3+z^4 & 2z+z^3+z^4 & 1+z^2+z^4 \,
\end{array}
\right).
\end{equation*}
Moreover, the first row of $\I+\B(z)+\B(z)^2+\B(z)^3+\B(z)^4$ has entries:
$3+z+z^2+z^3+z^4$, $1+z+z^2+z^3+2z^4$, $1+z+z^2+z^3+2z^4$, $1+z+z^2+z^3+2z^4$,
$2+z+z^2+z^3+z^4$.
Note that all coefficients $\alpha_i$, for $i=0,\ldots,4$, of these
polynomials are non-zero. This shows that the eccentricities of the vertices $(u,0)$, for $u\in V$, are 4, as expected since this is the diameter of $\Gamma^{\alpha}$.
By reading as columns, this means that if $\A$ is the adjacency matrix of the Alegre digraph, then, the first row of $\I+\A+\A^2+\A^3+\A^4$ is $3,1,1,1,2,1,1,\ldots,1,2,2,2,1$.

Note that, in the above result, the products of the entries (polynomials) of $\B(z)$ must be understood in the ring $\mathbb{R}_{k-1}[z]$. In terms of the vectors of the polynomial coefficients, this means that, in fact, we are carrying out a circular convolution. Thus, an efficient way of doing so is by using the discrete Fourier transform, since convolution in one variable is equivalent to entrywise product in the other.
In the next subsection, we give details and one example of this procedure.

\subsection{Using the discrete Fourier transform}
Recall that the {\em discrete Fourier transform} (DFT) of a vector $\z=(z_0,z_1,\ldots,z_{n-1})\in \C^n$ is the vector $\vecZ=(Z_0,Z_1,\ldots,Z_{n-1})\in \C^n$ with components
$$
Z_k=\frac{1}{\sqrt{n}}\sum_{\ell=0}^{n-1}z_{\ell}\omega^{-i k\ell},\qquad k=0,\ldots, n-1,
$$
where $\omega=e^{i\frac{2\pi}{n}}$ is the $n$-th root of the unity.
Then the inverse transform is
$$
z_{\ell}=\frac{1}{\sqrt{n}}\sum_{k=0}^{n-1}Z_{k}\omega^{i k\ell},\qquad \ell=0,\ldots, n-1,
$$
This is usually written as $\Z=\Fou\{\z\}=\F\z$ and $\z=\Fou^{-1}\{\z\}=\overline{\F}\z$, where
$\F$ is the matrix  with entries $f_{k\ell}=\frac{1}{\sqrt{n}}\omega^{-k\ell}$.
Here we use the property of the discrete Fourier transform
$$
\Fou(\z\ast\y)=\sqrt{n}(\Fou\{\z\}\circ\Fou\{\y\})
$$
where $\ast$ is the cyclic convolution, defined as
$$
(\z\ast\y)_i=\sum_{j=0}^{n-1}\z_j\y_{i-j(\!\!\!\!\!\!\mod n)},\qquad i=0,\ldots,n-1,
$$
and $\circ$ is the entrywise product of vectors
$$
(\z\circ \y)_i=\z_i\y_i,\qquad i=0,\ldots,n-1.
$$

\subsection*{The Alegre digraph revisited}
Let us see the application of this method by using again the Alegre digraph.
First, we consider the vectors of coefficients of the involved polynomials:
\begin{align*}
1\ &\rightarrow\ \z_0=(1,0,0,0,0);\\
z\ &\rightarrow\ \z_1=(0,1,0,0,0);\\
z^4\ &\rightarrow\ \z_4=(0,0,0,0,1);\\
z+z^4\ &\rightarrowº \z_{1,4}=\z_1+\z_4.
\end{align*}
with corresponding Fourier transforms (with entries rounded to three decimals):
\begin{align*}
\y_0 &=\Fou\{\z_0\}=(0.447,0.447,0.447,0.447,0.447);\\
\y_1 &=\Fou\{\z_1\}=(0.447,0.138-0.425i,-0.362-0.263i,-0.362+0.263i,0.138+0.425i); \\
\y_4 &=\Fou\{\z_4\}=(0.447,0.138+0.425i,-0.362+0.263i,-0.362-0.263i,0.138-0.425i); \\
\y_{1,4} &=\y_1+\y_4.
\end{align*}
This leads to the matrices:
$$
\Y_0 =\left(
\begin{array}{ccccc}
0& 0.447 & 0.447 & 0 & 0\\
0& 0.447 & 0.447 & 0 & 0\\
0& 0 & 0& 0.894 & 0\\
0.447& 0& 0& 0& 0.447\\
0.447& 0& 0& 0& 0.447
\end{array}
\right);
$$
$$
\Y_1 =\left(
\begin{array}{ccccc}
0& 0.447 & 0.447 & 0 & 0\\
0& 0.138+0.425i & 0.138+0.425i & 0 & 0\\
0& 0 & 0& 0.276 & 0\\
0.447& 0& 0& 0& 0.447\\
0.138-0.425i& 0& 0& 0& 0.138-0.425i
\end{array}
\right);
$$
$$
\Y_2 =\left(
\begin{array}{ccccc}
0& 0.447 & 0.447 & 0 & 0\\
0& -0.362+0.263i & -0.362+0.263i & 0 & 0\\
0& 0 & 0& -0.7236 & 0\\
0.447& 0& 0& 0& 0.447\\
-0.362-0.263i& 0& 0& 0& -0.362-0.263i
\end{array}
\right);
$$
$$
\Y_3 =\left(
\begin{array}{ccccc}
0& 0.447 & 0.447 & 0 & 0\\
0& -0.362-0.263i & -0.362-0.263i & 0 & 0\\
0& 0 & 0& -0.7236 & 0\\
0.447& 0& 0& 0& 0.447\\
-0.362+0.263i& 0& 0& 0& -0.362+0.263i
\end{array}
\right);
$$
$$
\Y_4 =\left(
\begin{array}{ccccc}
0& 0.447 & 0.447 & 0 & 0\\
0& 0.138-0.425i & 0.138-0.425i & 0 & 0\\
0& 0 & 0& 0.276 & 0\\
0.447& 0& 0& 0& 0.447\\
0.138+0.425i& 0& 0& 0& 0.138+0.425i
\end{array}
\right);
$$
Then, for instance, the $(0,0)$-entries of the matrices $\Y_0^4$, $\Y_1^4$ $\Y_2^4$, $\Y_3^4$, $\Y_4^4$
are the entries of the vector $\y_{0,0}=(0.160,0.015,0.105,0.105,0.015)$, with inverse Fourier
transform $(0.179,0,0.089,0.089,0)$. Hence, since $(\sqrt{5})^3(0.179,0,0.089,0.089,0)\approx
(2,0,1,1,0)$, we conclude that
$$
(\B(z)^4)_{0,0}=2+z^2+z^3,
$$
in concordance with \eqref{B(z)^4}.

\subsection{The case of non-cyclic groups}
In the case when $G$ is not a cyclic group, we can generalize the above matrix representation in the following manner:
Let $G$ be a group with generating set $\Delta=\{g_0,\ldots,g_{m-1}\}$. Given two vectors $\a,\b\in \Na^m$ (where
$\Na=\{0,1,2,\ldots\}$), their \emph{$G$-convolution} $\a\ast_{G}\b$ is defined to be the vector of
$\Na^m$ with components
$$
(\a\ast_{G}\b)_i=\sum_{j,k\,:\,g_jg_k=g_i} a_jb_k.
$$
Let $\G=(V,E)$ be a digraph with voltage assignment $\alpha$ on the group $G$. Its
\emph{$G$-representation matrix} $\B_G$ is a square matrix indexed by the vertices of $\G$, and whose
elements are vectors of $\Na^m$, $(\B_G)_{uv}=\b_{uv}$, where
$$
(\b_{uv})_i=\left\{
\begin{array}{ll}
1 & \mbox{if  $\exists\ uv\in E$ : $\alpha(uv)=g_i$,}\\
0 & \mbox{otherwise,}
\end{array}
\right.
$$
for $i=1,\ldots,m$.
The product of $\B_G$ by itself, denoted by $\B_G^2$, has entries
$$
\b_{uv}^{(2)}=\sum_{w\in V} \b_{uw}\ast_{G}\b _{wv},
$$
and the power matrix $\B_G^{\ell}$ is computed as expected.
The following result shows how such a power matrix contains the information about the walks in the lifted
digraph $\G^{\alpha}$.
\begin{lemma}
\label{lema-walks-general}
If the $uv$-entry of $\B_G^{\ell}$ is $\b_{uv}^{(\ell)}=(\beta_1,\beta_2,\ldots,\beta_n)$, then, for
every $i=1,\ldots,n$, there are $\beta_i$ walks of length $\ell$ from vertex $(u,h)$, $h\in G$, to
vertex $(v,hg_i)$ of the lifted graph $\G^{\alpha}$. \hfill $\Box$
\end{lemma}

In particular, if $G$ is an Abelian group, say $G=\Z_{k_1}\times\cdots\times \Z_{k_{n}}$, with $m=|G|=\prod_{i=1}^n k_i$, the vectors representing the entries of $\B_G$ can be replaced by polynomials with $n$ variables $z_1,\ldots,z_n$. Namely, $(\B_G)_{uv}=\sum_{i_1,\ldots,i_n}\alpha_{i_1,\ldots,i_n}z_1^{i_1}\cdots
z_n^{i_n}$, where
$$
\alpha_{i_1,\ldots,i_n}=\left\{
\begin{array}{ll}
1 & \mbox{if  $\exists\ uv\in E$ : $\alpha(uv)=(g_1,\ldots,g_m)\in G$,}\\
0 & \mbox{otherwise.}
\end{array}
\right.
$$
Then, as in the case of cyclic groups, we compute the powers  $\B_G^{\ell}$ using the standard polynomial multiplication, and the coefficients of the resulting polynomial entries gives the same information described in Lemmas \ref{lema-walks-cyclic} and \ref{lema-walks-general}. As an example of this case, in the next section we use the known fact that the Hoffman-Singleton graph can be constructed as a lift of a base graph on two vertices, with voltages in the group $\Z_5\times \Z_5$.

\section{The spectrum of the lifted digraph}
Apart from the obvious approach of computing the characteristic polynomial of the adjacency matrix, we aim to get a more simple  method of computing the whole spectrum of the lifted digraph $\G^{\alpha}$.
First, we have the following simple result, which is a consequence of Lemma \ref{lema-sp}.
\begin{corollary}
\label{coro-ev-lift}
Let $\G$ be a base digraph with vertices $u_1,\ldots,u_n$, and a given voltage assignment $\alpha$ on the
group $G$ with generating set $\{g_0,\ldots,g_{m-1}\}$. Let $\B=\sum_{i=0}^{m-1} \A_i$, where $\A_i$ is the adjacency matrix of the
subgraph of $\G$ with arc set $\alpha^{-1}(g_i)$. Then,
$$
\spec \B \subset \spec \G^{\alpha}.
$$
\vskip -1.1cm
\hfill  $\Box$
\end{corollary}

For example, the Alegre digraph has quotient matrix $\B=\B(1)$ given in \eqref{quotient-matrix-Alegre}, with spectrum $\spec \B=\{2,0^{(2)},i,-i\}$.
In fact, the spectrum of the Alegre digraph is
\begin{equation}
\label{sp-Alegre}
\spec \Gamma^\alpha=\left\{2,0^{(10)},i^{(5)},-i^{(5)}, \frac{1}{2} (-1+\sqrt{5})^{(2)},
\frac{1}{2}(-1-\sqrt{5})^{(2)}\right\},
\end{equation}
where, in agreement with Corollary \ref{coro-ev-lift}, we observe that $\spec \B\subset\spec
\Gamma^\alpha$.
Notice also that, in this case, the other eigenvalues of $\G$ are those ($\neq 2$) of the undirected cycle $C_5$, whose spectrum is
$$
\spec C_5=\left\{2,\frac{1}{2} (-1+\sqrt{5})^{(2)},
\frac{1}{2}(-1-\sqrt{5})^{(2)}\right\}.
$$

\subsection{The case of cyclic groups}
The following result shows how the spectrum of the lift $\G^{\alpha}$ can be completely determined from the spectrum of the polynomial matrix $\B(z)$ in the case when voltages are taken in a cyclic group. Here we assume that entries of $\B(z)$ are elements of the polynomial ring $\C(z)$. If $\Gamma$ has $r$ vertices and the cyclic group has order $k$, the characteristic polynomial $P(\lambda,z)={\rm det}(\lambda I - \B(z))$ is, technically, a polynomial in two complex variables $\lambda$ and $z$, of degree $r$ in $\lambda$ and at most $k-1$ in $z$. As we shall see, however, later we will be interested only in the corresponding polynomials in $\lambda$ arising by substituting suitable complex roots of unity for $z$.

\begin{proposition}
\label{propo-sp}
Let $\G=(V,E)$ be a base digraph on $r$ vertices, with a voltage assignment $\alpha$ in $\Z_k$. Let $P(\lambda,z)={\rm det}(\lambda I - \B(z))$ be the characteristic polynomial of the polynomial matrix $\B(z)$ of the voltage digraph $(\G,\alpha)$. For $j=0,\ldots,k-1$, let $\omega_j$  be the distinct $k$-th complex roots of unity. Then, the spectrum of the lift $\G^{\alpha}$ is the multiset of $kr$ roots $\lambda$ of the $k$ polynomials $P(\lambda,\omega_j)$ of degree $r$ each, where $0\le j\le k-1$; formally,
$$
\spec \G^{\alpha}=\{\lambda_{i,j}\, :\, P(\lambda_{i,j},\omega_j)=0,\ 1\le i\le r,\ 0\le j\le k-1\}.
$$
\end{proposition}
\begin{proof}
Although entries of $\B(z)$ are polynomials in a complex variable $z$, in what follows let $z$ be any fixed complex number for which we will make appropriate choices later. For our $z\in \C$, let $\x(z)=(x_u(z))_{u\in V}$ be an eigenvector corresponding to an eigenvalue $\lambda(z)$ of the matrix $\B(z)$ of $\G$; that is,
\begin{equation}
\label{eq-ev-quotient}
\sum_{uv\in E}z^{\alpha(uv)}x_v(z)=\lambda(z)x_{u}(z).
\end{equation}
Let $R(k)$ denote the set of $k$-th roots of unity. Making now the choice $z=\omega\in R(k)$ in \eqref{eq-ev-quotient} we obtain
$$
\sum_{uv\in E}\omega^{\alpha(uv)}x_v(\omega)=\lambda(\omega)x_{u}(\omega).
$$
Multiplying by $\omega^{i}$ for any (fixed) $i\in \Z_{k}$, we have
\begin{equation}
\label{eq-ev-quotient2}
\sum_{uv\in E}\omega^{i+\alpha(uv)}x_v(\omega)=\lambda(\omega)\omega^i x_{u}(\omega).
\end{equation}
Now, for every pair $(u,j)\in V\times \Z_k$, let the map $\phi_{(u,j)}:R(k)\rightarrow \C$
be defined as
$$
\phi_{(u,j)}(\omega)=\omega^j x_u(\omega).
$$
Then, as $\phi_{(v,i+\alpha(uv))}(\omega)=\omega^{i+\alpha(uv)}x_v(\omega)$, we can rewrite
\eqref{eq-ev-quotient2} in the form
$$
\sum_{uv\in E} \phi_{(v,i+\alpha(uv))}(\omega)=\lambda(\omega)\phi_{(u,i)}(\omega).
$$
But this means that $\lambda(\omega)$ is an eigenvalue of the lift, corresponding to the eigenvector
$\phi(\omega):= (\phi_{(u,i)}(\omega))_{(u,i)\in V\times \Z_k}$.

Since $\lambda=\lambda(\omega)$ is a root of the characteristic polynomial $P(\lambda,\omega)$,  we
obtain, in this way, a total of $rk$ eigenvalues (including repetitions), which is the number of eigenvalues of the adjacency matrix $\A$ of the lift $\G^{\alpha}$.

According to the properties of the polynomial matrix, if
$$
(\B(z)^{\ell})_{uu}=\alpha_{u,0}^{(\ell)}+\alpha_{u,1}^{(\ell)} z+\cdots +\alpha_{u,k-1}^{(\ell)}z^{k-1}
$$
then, the total number of rooted closed $\ell$-walks in $\G^{\alpha}$ is
$$
\tr (\A^\ell)=\sum_{\lambda\in \spec \G^{\alpha}} \lambda^\ell=k\sum_{u\in V} \alpha_{u,0}^{(\ell)}.
$$
But, since $\sum_{j=0}^{k-1}\omega_j^{\ell}=0$ for every $j,\ell\neq 0$, we have that
$$
\alpha_{u,0}^{(\ell)}=\frac{1}{k}\sum_{j=0}^{k-1} (\B(\omega_j))_{uu}.
$$
Then,
$$
\sum_{\lambda\in \spec \G^{\alpha}} \lambda^\ell=\tr (\A^\ell)
=\sum_{u\in V}\sum_{j=0}^{k-1} (\B(\omega_j))_{uu}
=\sum_{j=0}^{k-1}\tr (\B(\omega_j))
=\sum_{j=0}^{k-1}\sum_{\mu\in \spec \B(\omega_j)} \mu^{\ell}.
$$
Since this is true for any value of $\ell\ge 0$, both (multi)sets of eigenvalues must coincide (see for example Gould  \cite{go99}).
\end{proof}

Returning to the example of the Alegre digraph, its polynomial matrix $\B(z)$ in \eqref{B(z)Alegre}
has eigenvalues $0$, with multiplicity 2, and $i$, $-i$, and $z+\frac{1}{z}$ with multiplicity 1.
Then, from Proposition \ref{propo-sp}, evaluating them at the $5$-th roots of unity $\omega_i$, for $i=0,1,2,3,4$, we get the complete spectrum \eqref{sp-Alegre} of the digraph, see Table \ref{evs-Alegre}.


\begin{table}[h]
\centering
\begin{tabular}{|c||l|l|l|l|}
\hline
  $z\backslash \lambda(z)$    & $0^{(2)}$ & $i^{(1)}$ & $-i^{(1)}$ & $(z+\frac{1}{z})^{(1)}$\\[.1cm]
\hline\hline
1 &  $0^{(2)}$ & $i^{(1)}$ & $-i^{(1)}$ & $2^{(1)}$ \\[.1cm]
\hline
$\omega$ &   $0^{(2)}$ & $i^{(1)}$ & $-i^{(1)}$ & $\frac{1}{2} (-1+\sqrt{5})^{(1)}$\\[.1cm]
\hline
$\omega^2$ &   $0^{(2)}$ & $i^{(1)}$ & $-i^{(1)}$ & $\frac{1}{2} (-1-\sqrt{5})^{(1)}$   \\[.1cm]
\hline
$\omega^3$ &   $0^{(2)}$ & $i^{(1)}$ & $-i^{(1)}$& $\frac{1}{2} (-1+\sqrt{5})^{(1)}$   \\[.1cm]
\hline
$\omega^4$ &   $0^{(2)}$ & $i^{(1)}$ & $-i^{(1)}$ & $\frac{1}{2} (-1-\sqrt{5})^{(1)}$   \\[.1cm]
 \hline
\end{tabular}
\caption{The eigenvalues of the Alegre digraph.}
\label{evs-Alegre}
\end{table}

\subsection{The digraphs ${\cal P}(n,p_1,p_2)$}
As another example of application, consider the following family of digraphs, which  contains, as a particular case, the well-known generalized Petersen graphs; see for example Gera and St\v{a}nic\v{a} \cite{gs11}. (Notice that all of our results apply also to graphs, since they can be just considered as symmetric digraphs where each digon represents an edge.)
Given an integer $n$ and two polynomials $p_1,p_2\in \Re_{n-1}[z]$, the digraph ${\cal P}(n,p_1,p_2)$
is obtained by `cyclically joining' $n$ (undirected) edges. More precisely, ${\cal P}(n,p_1,p_2)$ is the lift of the base digraph with polynomial matrix
$$
\B(z)=\left(
\begin{array}{cc}
p_1(z) & 1 \\
1 & p_2(z)
\end{array}
\right).
$$
Then, from Proposition \ref{propo-sp}, the eigenvalues of ${\cal P}(n,p_1,p_2)$ are
\begin{equation}
\label{ev-P(p1,p2)}
\lambda_{i,j}=\frac{1}{2}\left[ p_1(\omega_j)+ p_2(\omega_j)\pm \sqrt{ (p_1(\omega_j)+ p_2(\omega_j))^2+4}\right], \ i=0,1,\  j=0,\ldots,n-1,
\end{equation}
where $\omega_j$ is the $j$-th $n$-root of unity.
This includes the results of  Gera and St\v{a}nic\v{a} \cite{gs11} about the spectra of the generalized Petersen graphs, denoted by $P(n,k)$, where $n\in N$ and $k\in \Z_n$.
Recall that $P(n,k)$ has $2n$ vertices, $u_0,\ldots,u_{n-1}$ and $v_0,\ldots,v_{n-1}$, with edges $(u_i,v_i)$, $(u_i,u_{i\pm 1})$, and $(v_i,v_{i\pm k})$ (arithmetic understood modulo $n$). Then, in our context, $P(n,k)$ corresponds to the case when we take the polynomials $p_1(z)=z+z^{-1}$ and $p_2(z)=z^k+z^{-k}$ .
Thus, when we evaluate these polynomials at $\omega_j$, we get
$$
p_1(w_j)=2\cos \left(\frac{j 2\pi}{n} \right)\qquad \mbox{and}\qquad p_2(w_j)=2\cos \left(\frac{j k 2\pi}{n}\right),
$$
and \eqref{ev-P(p1,p2)} yields the result in \cite{gs11} (Theorem 2.4 and Corollary 2.5).

\subsection{The Hoffman-Singleton graph}

As commented in the previous section, when the group $G$ is Abelian, the lift can be constructed similarly by using multivariate polynomials as the entries of $\B_G$. Moreover the above spectral theory also works with the natural changes. Instead of `repeating ourselves' by giving the analogous of  Proposition \ref{propo-sp}, we follow an example with the Hoffman-Singleton (HS) graph, first discovered in \cite{hs60}. As it is common knowledge, this  is a Moore $7$-regular graph with $50$ vertices, and diameter $2$ (for more details, see e.g. Hoffman and Singleton \cite{hs60} or  Godsil \cite{g93}). As it was shown by \v{S}iagiov\'a  \cite{s01} (see also Mirka and \v{S}ir\'a\v{n} \cite{ms}), the HS graph can be obtained as the lift of a base graph $G$ consisting of two vertices with voltages in the group $\Z_5\times \Z_5$, see Figure \ref{fig4}, and polynomial matrix
$$
\B(w,z)=\left(
\begin{array}{cc}
w+\frac{1}{w} & 1+zw+z^2w^4+z^3w^4+z^4w^4 \\
1+\frac{1}{zw}+\frac{1}{z^2w^4}+\frac{1}{z^3w^4}+\frac{1}{z^4w^4} & w^2+\frac{1}{w^2}
\end{array}
\right).
$$

\begin{figure}[t]
\begin{center}
\includegraphics[width=8cm]{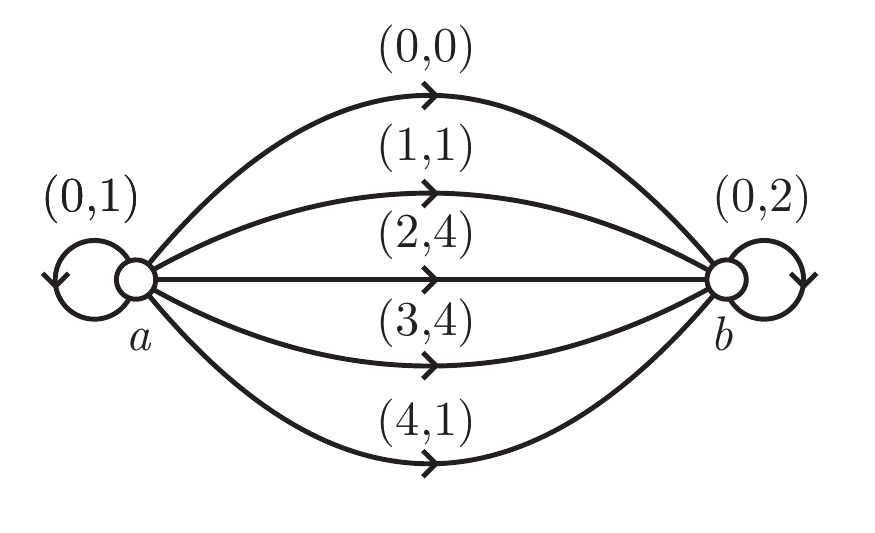}
\end{center}
\vskip -.8cm
\caption{A base graph for the HS graph.}
  \label{fig4}
\end{figure}

Then, by giving to $(w,z)$ the 25 possible values in $R(5)\times R(5)$, the eigenvalues of $\B(w,z)$ are shown in Table \ref{evs-HS}.

\begin{table}[h]
\centering
\begin{tabular}{|c||l|l|l|l|l|}
\hline
  $z\backslash w$    & 1        & $\omega$ & $\omega^2$ & $\omega^3$ & $\omega^4$ \\
\hline\hline
1           &  $7,-3$  & $2,-3$   &  $2,-3$    &  $2,-3$   &  $2,-3$    \\
\hline
$\omega$     &  $2,2$  & $2,-3$   &  $2,-3$    &  $2,-3$   &  $2,-3$  \\
\hline
$\omega^2$  &  $2,2$  & $2,-3$   &  $2,-3$    &  $2,-3$   &  $2,-3$  \\
\hline
$\omega^3$  &  $2,2$  & $2,-3$   &  $2,-3$    &  $2,-3$   &  $2,-3$   \\
\hline
 $\omega^4$  &  $2,2$  & $2,-3$   &  $2,-3$    &  $2,-3 $   &  $2,-3$  \\
 \hline
\end{tabular}
\caption{The eigenvalues of the HS graph.}
\label{evs-HS}
\end{table}

As a consequence, we have that the spectrum of the HS graph is
$$
\spec {\rm HS}=\{7^{(1)}, 2^{(28)}, -3^{(21)}\},
$$
as it is well known, as the HS graphs is strongly regular.


\end{document}